\documentclass[12pt]{article}

\usepackage[margin=1in]{geometry}  % set the margins to 1in on all sides

\usepackage{helvet}         % selects Helvetica as sans-serif font
\usepackage{courier}        % selects Courier as typewriter font
\usepackage{type1cm}        % activate if the above 3 fonts are
\usepackage{graphicx}        % standard LaTeX graphics tool
\usepackage[bottom]{footmisc}% places footnotes at page bottom

\usepackage[utf8]{inputenc}
\usepackage[T1]{fontenc}
\usepackage{amsmath}
\usepackage{amssymb}
\usepackage{amsthm}
\usepackage[english]{babel}

\usepackage{cite}

\usepackage{enumitem}  % http://ctan.org/pkg/enumitem
\usepackage{calc}% http://ctan.org/pkg/calc
\setlist{labelindent=1pt,itemsep=.5em}
\setlist[itemize]{leftmargin=1.2cm}
\setlist[enumerate]{itemindent=0em,leftmargin=1.2cm}
\setlist[enumerate,1]{label={\upshape(\roman*)}}

\usepackage{authblk}

\makeatletter
\newcommand{\subjclass}[2][2020]{%
  \let\@oldtitle\@title%
  \gdef\@title{\@oldtitle\footnotetext{#1 \emph{Mathematics subject classification}: #2}}%
}
\newcommand{\keywords}[1]{%
  \let\@@oldtitle\@title%
  \gdef\@title{\@@oldtitle\footnotetext{\emph{Keywords}: #1}}%
}
\makeatother

\allowdisplaybreaks[3]

\newcommand*\sq{\mathbin{\vcenter{\hbox{\rule{.3ex}{.3ex}}}}}

\newtheorem{theorem}{Theorem}[section]
\newtheorem{lemma}[theorem]{Lemma}

\theoremstyle{definition}
\newtheorem{definition}[theorem]{Definition}
\newtheorem{example}[theorem]{Example}

\theoremstyle{remark}

\title{HNN-extension of involutive multiplicative Hom-Lie algebras}

\author{Sergei Silvestrov$^{1}$, Chia Zargeh$^{2}$ \\
\small{
$^{1}$ Division of Mathematics and Physics,
School of Education, Culture and Communication,
M\"{a}lardalen University, Box 883, 72123 V\"{a}steras, Sweden. \authorcr
e-mail: sergei.silvestrov@mdh.se} \authorcr
$^{2}$Instituto de Matem\'atica e Estat\'istica, Universidade de S\~ao Paulo, S\~ao Paulo, Brazil. \authorcr
e-mail: chia.zargeh@ime.usp.br}


\date{}

\begin{document}

\maketitle

\abstract{The construction of HNN-extensions of involutive Hom-associative algebras and involutive Hom-Lie algebras is described. Then, as an application of HNN-extension, by using the validity of Poincar\'e-Birkhoff-Witt theorem for involutive Hom-Lie algebras, we provide an embedding theorem.}

\section{Introduction}

One of the most important constructions in combinatorial group theory is Higman-Neumann-Neumann extension (or HNN-extension, for short), which states that if $A_1$ and $A_2$ are isomorphic subgroups of a group $G$, then it is possible to find a group $H$ containing $G$ such that $A_1$ and $A_2$ are conjugate to each other in $H$ and $G$ is embeddable in $H$ (see\cite{HigmanNeumannBHNeumannH49:Embedtheogroups}). The HNN-extension of a group has a topological interpretation described in \cite{BokutKukun94:algortmcombalg} and  \cite{LyndonSchup:combinGrTheory1977}, which is used as a motivation for its study. Spreading classical techniques in combinatorial group theory to other algebraic structures has shown outstanding capacities for solving problems in affine algebraic geometry, the theory of Lie algebras and mathematical physics. In this regard, HNN-extension of Lie algebras was constructed by Lichtman and Shirvani \cite{LichtmanShirvani1997:HNNextLiealg} and Wasserman \cite{Wasserman98:derivHNNconstrLiealg} through different approaches. They used HNN-extension in order to give a new proof for Shirshov’s theorem \cite{Shirshov58:OnfreeLierings}, namely, a Lie algebra of finite or countable dimension can be embedded into a $2$-generator Lie algebra. Moreover, the idea of HNN-extension has been recently spread to Leibniz algebras in \cite{LandraShahryariZardeh2019:HNNextLeibnalg} and Lie superalgebras in \cite{LandraPaezGuillanZargeh2020:HNNextLeibnizalg}, which are respectively, non-antisymmetric and natural generalization of Lie algebras.

In this paper we intend to introduce HNN-extension for the Hom-generali\-zation of Lie algebras.
Hom-Lie algebras and more general quasi-Hom-Lie algebras were introduced first by Hartwig, Larsson and Silvestrov in \cite{HartwigLarssonSilvestrov:defLiealgsderiv}, where the general quasi-deformations and discretizations of Lie algebras of vector fields using more general $\sigma$-derivations (twisted derivations) and a general method for construction of deformations of Witt and Virasoro type algebras based on twisted derivations have been developed, initially motivated by the $q$-deformed Jacobi identities observed for the $q$-deformed algebras in physics,  $q$-deformed versions of homological algebra and discrete modifications of differential calculi. Hom-Lie superalgebras, Hom-Lie color algebras and more general quasi-Lie algebras and color quasi-Lie algebras where introduced first in \cite{LarssonSilv2005:QuasiLieAlg,LarssonSilv:GradedquasiLiealg,SigSilv:CzechJP2006:GradedquasiLiealgWitt}. Quasi-Lie algebras and color quasi-Lie algebras encompass within the same algebraic framework the quasi-deformations and discretizations of Lie algebras of vector fields by $\sigma$-derivations obeying twisted Leibniz rule, and color Lie algebras, the well-known natural generalizations of Lie algebras and Lie superalgebras. In quasi-Lie algebras, the skew-symmetry and the Jacobi identity are twisted by deforming twisting linear maps, with the Jacobi identity in quasi-Lie and quasi-Hom-Lie algebras in general containing six twisted triple bracket terms. In Hom-Lie algebras, the bilinear product satisfies the non-twisted skew-symmetry property as in Lie algebras, and the Hom-Lie algebras Jacobi identity has three terms twisted by a single linear map, reducing to the Lie algebras Jacobi identity when the twisting linear map is the identity map. Hom-Lie admissible algebras have been considered first in \cite{ms:homstructure}, where in particular the Hom-associative algebras have been introduced and shown to be Hom-Lie admissible, leading to Hom-Lie algebras using commutator map as new product, and thus  constituting a natural generalization of associative algebras as Lie admissible algebras. Since the pioneering works \cite{HartwigLarssonSilvestrov:defLiealgsderiv,LarssonSilvJA2005:QuasiHomLieCentExt2cocyid,LarssonSilv:GradedquasiLiealg,LarssonSilv2005:QuasiLieAlg,LarssonSilv:QuasidefSl2,ms:homstructure}, Hom-algebra structures expanded into a popular area with increasing number of publications in various directions. Hom-algebra structures of a given type include their classical counterparts and open broad possibilities for deformations, Hom-algebra extensions of cohomological structures and representations, formal deformations of Hom-associative algebras and Hom-Lie algebras, Hom-Lie admissible Hom-coalgebras, Hom-coalgebras, Hom-Hopf algebras,
Hom-Lie algebras, Hom-Lie superalgebras, color Hom-Lie algebras, BiHom-Lie algebras, BiHom-associative algebras, BiHom-Frobenius algebras and $n$-ary generalizations of Hom-algebra structures have been further investigated in various aspects for example in \cite{AbdaouiAmmarMakhloufCohhomLiecolalg2015,Abdelkader2017:generalizedderivBiHomalgebras,AbramovSilvestrov:3homLiealgsigmaderivINvol,AmmarEjbehiMakhlouf:homdeformation,
AmmarMabroukMakhloufCohomnaryHNLalg2011,AmmarMakhloufHomLieSupAlg2010,AmmarMakhloufSaadaoui2013:CohlgHomLiesupqdefWittSup,
AmmarMakhloufSilv:TernaryqVirasoroHomNambuLie,ArmakanFarhangdoost:IJGMMP,ArmakanSilv:envelalgcertaintypescolorHomLie,ArmakanSilvFarh:envelopalgcolhomLiealg,
ArmakanSilvFarh:exthomLiecoloralg,ArmakanSilv:NondegKillingformsHomLiesuperalg,akms:ternary,ams:ternary,ArnlindMakhloufSilvnaryHomLieNambuJMP2011,
Bakayoko2014:ModulescolorHomPoisson,Bakayoko:LaplacehomLiequasibialg,Bakayoko:LmodcomodhomLiequasibialg,BakayokoDialo2015:genHomalgebrastr,BakyokoSilvestrov:Homleftsymmetriccolordialgebras,
BakyokoSilvestrov:MultiplicnHomLiecoloralg,BakayokoToure2019:genHomalgebrastr,BenAbdeljElhamdKaygorMakhl201920GenDernBiHomLiealg,
BenHassineMabroukNcib:ConstrMultiplicnaryhomNambualg,BenHassineChtiouiMabroukNcib:Strcohom3LieRinehartsuperalg,BenMakh:Hombiliform,
CanepeelGoyaverts:MonoidalHomHopfalgebras,CaoChen2012:SplitregularhomLiecoloralg,ChengQi2016:RepresentBiHomLiealg,ElchingerLundMakhSilv:BracktausigmaderivWittVir,
GrazMakhlMeniniPanaite:bihom,GuanChenSun:HomLieSuperalgebras,HeMaSiUnAlHomAss,HounkonnouHoundedjiSilvestrov:DoubleconstrbiHomFrobalg,
KitouniMakhloufSilvestrov,kms:solvnilpnhomlie2020,kms:narygenBiHomLieBiHomassalgebras2020,
LarssonSigSilvJGLTA2008:QuasiLiedefFttN,LarssonSilvestrovGLTMPBSpr2009:GenNComplTwistDer,
MabroukNcibSilvestrov2020:GenDerRotaBaxterOpsnaryHomNambuSuperalgs,ms:homstructure,MakhSilv:HomDeform,MakhSil:HomHopf,MakhSilv:HomAlgHomCoalg,
Makhlouf2010:ParadigmnonassHomalgHomsuper,MandalMishra:HomGerstenhaberHomLiealgebroids,MishraSilvestrov:SpringerAAS2020HomGerstenhalgsHomLiealgds,
RichardSilvJA2008:quasiLiesigderCtpm1,RichardSilvestrovGLTMPBSpr2009:QuasiLieHomLiesigmaderiv,Saadaoui:ClassmultiplsimpleBiHomLiealg,
SigSilv:CzechJP2006:GradedquasiLiealgWitt,Sheng:homrep,ShengBai2014:homLiebialg,ShengChen2013:HomLie2algebras,ShengXiong:LMLA2015:OnHomLiealg,SigSilv:GLTbdSpringer2009,SilvestrovParadigmQLieQhomLie2007,
Yau2009:HomYangBaxterHomLiequasitring,Yau:EnvLieAlg,Yau:HomolHom,Yau:HomBial,Yuan2012:HomLiecoloralgstr,
ZhouNiuChen:GhomDerivation,ZhouChenMa:GenDerHomLiesuper,ZhouZhaoZhang:GenDerHomLeibnizalg}.

Our approach for construction of the HNN-extension of Hom-gene\-ralization of Lie algebras is based on the corresponding construction for its envelope. Therefore, we concentrate on the study of HNN-extensions for involutive Hom-Lie algebras in which their universal enveloping algebras have been explicitly obtained in \cite{GuoZhZheUEPBWHLieA}.
It is worth noting that there exists another approach provided in \cite{Yau:EnvLieAlg} for obtaining the universal enveloping algebra of a Hom-Lie algebra as a suitable quotient of the free  Hom-nonassociative algebra through weighted trees, but the point of difficulty in the approach in \cite{Yau:EnvLieAlg} is the size of the weighted trees. Involutive Hom-Lie algebras have been constructed in \cite{ZhengGuo16:FreeinvolHomsemigrHomassalg}, and the classical theory of enveloping algebras of Lie algebras was extended to an explicit construction of the free involutive Hom-associative algebra on a Hom-module in order to obtain the universal enveloping algebra \cite{GuoZhZheUEPBWHLieA}. This construction leads to a Poincare-Birkhoff-Witt theorem for the enveloping associative algebra of an involutive Hom-Lie algebra. This approach has been extended to the enveloping algebras for color Hom-Lie algebras in \cite{ArmakanSilvFarh:envelopalgcolhomLiealg,ArmakanSilv:envelalgcertaintypescolorHomLie}. Extensions of Hom-Lie superalgebras and Hom-Lie color algebras have been considered in \cite{ArmakanFarhangdoost:IJGMMP,ArmakanSilvFarh:exthomLiecoloralg}. Hom-associative Ore extensions have been considered in \cite{Back2018:HomassociativeOreextensions,Back2018:HomassOreextsweakunit,Back2018:HilbertbasisthmnonassHomassOreexts,Back2019:formaldefsquantplanesunivenvalgs,
Back2020:MultiparformaldefsternhomNambuLiealgs,Back2020:homassociativeWeylalgebras}

The paper is organized as follows. In Section \ref{sec:InvolHomLie}, we recall the preliminary concepts related to involutive Hom-associative algebras and involutive Hom-Lie algebras. In Section \ref{sec:HNNextensioninvolHomassocalgs}, we introduce the HNN-extension for involutive Hom-associative algebras. In Section \ref{sec:HNNextensioninvolHomLiealgs}, we construct the HNN-extension for involutive Hom-Lie algebras and provide an embedding theorem.

\section{Involutive Hom-algebras} \label{sec:InvolHomLie}
In this section we recall necessary concepts related to involutive Hom-associative and involutive Hom-Lie algebras.
\begin{definition} \label{def-involutive}
Let $K$ be a field.
\begin{enumerate}[label=\upshape{(\alph*)},left=0pt]
    \item Hom-module is a pair $(V,\alpha_V)$ consisting of a $K$-module $V$ and a linear operator $\alpha_V:V \to V$.
   \item Hom-associative algebra is a triple $(A, \ast_A ,\alpha_A)$ consisting of a $K$-module $A$, a linear map~$ \ast_A :A \otimes A \to A$, called the multiplication, and a linear operator $\alpha_A :A\to A$ satisfying the Hom-associativity
   \[  \alpha_A(x)\ast_A (y\ast_A z)=(x\ast_A y)\ast_A \alpha_A (z), \]
   for all $x,y,z \in A.$
   \item Hom-associative algebra is said to be {\em multiplicative} if the linear map $\alpha$ is multiplicative in the sense of satisfying $\alpha_A(x\ast_A y)= \alpha_A(x) \ast_A \alpha_A (y)$ for all $x, y \in A$.
   \item Hom-associative algebra $(A, \ast_A ,\alpha_A)$ (resp.  Hom-module $(V,\alpha_V)$) is said to be {\em involutive} if $\alpha^{2}_A=id$ (resp. $\alpha^{2}_V=id)$.
   \item Let  $(V,\alpha_V)$ and  $(W,\alpha_W)$  be  Hom-modules.   A $K$-linear  map $f:V \to W$ is  called  a morphism of Hom-modules if $f(\alpha_V (x))=\alpha_W(f(x)) $ for all $x \in V.$
   \item Let $(A, \ast_A ,\alpha_A)$ and $(B, \ast_B ,\alpha_B)$ be two Hom-associative algebras.  A $K$-linear map $f:A \to B$ is a morphism of Hom-associative algebras if
   \[ f(x \ast_A y)=f(x) \ast_B f(y),~ \text{and}~ f(\alpha_A(x))=\alpha_B(f(x)), \]
   for all $x,y \in A.$
   \item  Let  $(A, \ast_A ,\alpha_A)$ be  a  Hom-associative  algebra.   A  submodule $B\subseteq A$ is called a Hom-associative subalgebra of $A$ if $B$ is closed under the multiplication $\ast_A$ and $\alpha_A(B)\subseteq B$.
   \item  Let $(A, \ast_A ,\alpha_A)$ be a Hom-associative algebra. A submodule $I \subseteq A$ is called a Hom-ideal of $A$ if $x\ast_A y  \in I$ , $y \ast_A x \in I$ for all $x\in I,$ $y \in A$, and $\alpha_A(I) \subseteq I$.
\end{enumerate}
\end{definition}
\begin{definition}\label{derivation}
For any non-negative integer $k$, a linear map $D:A \to A$ is called an $\alpha_{A}^{k}$-derivation of involutive Hom-associative algebra $(A,\ast_A,\alpha_A)$, if
\begin{eqnarray*}
D \circ \alpha_{A}^{k} &=& \alpha_{A}^{k} \circ D, \\
D \circ (x \ast_A y) &=& D(x)\ast_A \alpha_{A}^{k}(y) + \alpha_{A}^{k}(x) \ast_A D(y).
\end{eqnarray*}
\end{definition}
\begin{definition}
Let $(V,\alpha_V)$ be an involutive Hom-module.  A free involutive Hom-associative algebra on $V$ is an involutive Hom-associative algebra $(F_{IH A}(V),\ast_F,\alpha_F)$ together with a morphism of Hom-modules $j_V:  (V,\alpha_V)\to (F_{IHA}(V),\alpha_F)$ such
 that, for any involutive Hom-associative algebra $(A,\ast_A,\alpha_A)$ together with a morphism of Hom-modules $f:  (V,\alpha_V)\to (A,\alpha_A)$, there is a unique morphism of Hom-associative algebras
$f:(F_{IHA}(V),\ast_F,\alpha_F)\to (A,\ast_A,\alpha_A)$ such that $f=f \circ j_V$.
\end{definition}

 \begin{definition}
 A Hom-Lie algebra is a triple $(\mathfrak{g}, [\cdot,\cdot]_{\mathfrak{g}},\beta)$ consisting of a vector space $\mathfrak{g}$, a skew-symmetric bilinear map (bracket) $[\cdot,\cdot]_{\mathfrak{g}}: \mathfrak{g} \times \mathfrak{g} \to \mathfrak{g}$ and a linear map $\beta: \mathfrak{g} \to \mathfrak{g}$ satisfying the following Hom-Jacobi identity:
 \begin{equation}\label{Hom-Jacobi identity}
     [\beta(u),[v, w]_{\mathfrak{g}}]_{\mathfrak{g}}+[\beta(v),[w,u]_{\mathfrak{g}}]_{\mathfrak{g}}+[\beta(w),[u,v]_{\mathfrak{g}}]_{\mathfrak{g}}=0.
 \end{equation}
Hom-Lie algebra is called a {\it multiplicative} Hom-Lie algebra if $\beta$ satisfies
\begin{equation}
    \beta([u,v]_{\mathfrak{g}})=[\beta(u), \beta(v)]_{\mathfrak{g}}.
\end{equation}
\end{definition}
A Hom-Lie algebra $(\mathfrak{g}, [\cdot,\cdot]_{\mathfrak{g}},\beta)$ is called {\it involutive} if $\beta^{2}=id_{\mathfrak{g}}$.
Note that the classical Lie algebra can be recovered when $\beta=id_{\mathfrak{g}}$, with the identity \eqref{Hom-Jacobi identity} becoming the Jacobi identity for Lie algebras.
 \begin{definition}
 A morphism of Hom-Lie algebras $$f: (\mathfrak{g},{[\cdot,\cdot]}_\mathfrak{g},\beta_\mathfrak{g}) \to (\mathfrak{h},[\cdot,\cdot]_\mathfrak{h},\beta_\mathfrak{h})$$ is a $k$-linear map $f:\mathfrak{g}\to \mathfrak{h}$ such that
 $$f([x,y]_{\mathfrak{g}})=[f(x),f(y)]_\mathfrak{h} \text{ and } f(\beta_{\mathfrak{g}}(x))=\beta_{\mathfrak{h}}(f(x)) \text{ for all } x\in \mathfrak{g}.$$
 \end{definition}

Hom-associative algebras were introduced in \cite{ms:homstructure}, and shown to be Hom-Lie admissible, i.e. any Hom-associative algebra $(A,\ast_A,\alpha_A)$ yields a Hom-Lie algebra $(A, [\cdot,\cdot]_{A},\beta_A)$ with $\beta_A=\alpha_A$ and $[x,y]_A=x\ast_A y-y\ast_A x$ for $x,y \in A$.

For simplicity, we will restrict our considerations to multiplicative Hom-Lie algebras and multiplicative Hom-associative algebras, meaning that the twisting map is not only linear, but also an endomorphism of the Hom-Lie algebra or Hom-associative algebra respectively.
An interesting important problem is to understand completely the role of the multiplicatives restriction and extend the results and constructions from multiplicative to general, not necessarily multiplicative, Hom-Lie algebras and Hom-associative algebras.

\begin{definition}[\cite{GuoZhZheUEPBWHLieA}]
 Let $(\mathfrak{g}, {[\cdot,\cdot]}_\mathfrak{g},\beta)$ be a Hom-Lie algebra. A universal enveloping Hom-associative algebra of $\mathfrak{g}$ is a Hom-associative algebra ${\mathfrak{U}}_\mathfrak{g}=({\mathfrak{U}}_{\mathfrak{g}},{\ast}_\mathfrak{g}, {\alpha}_{\mathfrak{U}})$, together with a morphism $\phi_{\mathfrak{g}}: (\mathfrak{g}, [\cdot,\cdot]_\mathfrak{g},\beta) \to ({\mathfrak{U}}_\mathfrak{g},{[\cdot,\cdot]}_{\mathfrak{U}_\mathfrak{g}}, {\beta_\mathfrak{U}}_{\mathfrak{g}})$ of Hom-Lie algebras, that satisfies the universal property.
 \end{definition}
 The following lemma describes the universal property in the involutive case.
 \begin{lemma}[\cite{GuoZhZheUEPBWHLieA}] \label{invouniversal}
 Let $(\mathfrak{g}, [\cdot,\cdot]_{\mathfrak{g}},\beta_{\mathfrak{g}})$ be an involutive multiplicative Hom-Lie algebra.
 \begin{enumerate}[label=\upshape{(\alph*)},left=0pt]
     \item Let $(A, \ast_A, \alpha_A)$ be a multiplicative Hom-associative algebra,
     $$f: (\mathfrak{g}, [\cdot,\cdot]_{\mathfrak{g}},\beta_{\mathfrak{g}}) \to (A, [\cdot,\cdot]_A, \beta_A)$$
     be a morphism of Hom-Lie algebras, and $B$ be the multiplicative Hom-associative subalgbera of $A$ generated by $f(\mathfrak{g})$. Then $B$ is involutive.
     \item The universal enveloping multiplicative Hom-associative algebra $(\mathfrak{U}_{\mathfrak{g}}, \phi_{\mathfrak{g}})$ of \\ $(\mathfrak{g},[\cdot,\cdot]_{\mathfrak{g}},\beta_{\mathfrak{g}})$ is involutive.
     \item In order to verify the universal property of $(\mathfrak{U}_{\mathfrak{g}}, \phi_{\mathfrak{g}})$, we only need to consider involutive multiplicative Hom-associative algebras $A:= (A, {\ast}_{A},\alpha_{A})$.
 \end{enumerate}
 \end{lemma}

 \begin{definition}
A linear subspace $\mathfrak{s} \subseteq \mathfrak{g}$ is called a Hom-Lie subalgebra of a Hom-Lie algebras $(\mathfrak{g}, [\cdot,\cdot]_{\mathfrak{g}},\beta)$ if $\beta(\mathfrak{s})\subseteq \mathfrak{s}$ and $\mathfrak{s}$ is closed under the bracket operation $[\cdot,\cdot]_{\mathfrak{g}}$:
\[\forall s_1, s_2 \in \mathfrak{s}: \quad [s_1, s_2]_{\mathfrak{g}} \in \mathfrak{s}.\]
\end{definition}

Let $(\mathfrak{g},[\cdot,\cdot]_{\mathfrak{g}},\beta)$ be a multiplicative Hom-Lie algebra. For any nonnegative integer $k$,denote by $\beta^{k}$ the $k$-times composition of $\beta$, i.e.
\[\beta^{k}=\beta \dots \beta ~(k\text{-times}).\]
In particular, $\beta^{0}=Id$ and $\beta^{1}=\beta$.

%\begin{definition}
%A linear map $D: \mathfrak{g} \to \mathfrak{g}$ i called a derivation of a Hom-Lie algebra $(\mathfrak{g},[\cdot,\cdot]_{\mathfrak{g}},\alpha)$ if the following identity holds:
%\[ D[x,y]_{\mathfrak{g}}=[\alpha(x),(Ad_{\alpha_{\mathfrak{g}}^{-1}} D)(y)]_{\mathfrak{g}}+[(Ad_{\alpha_{\mathfrak{g}}}^{-1} D)(x),\alpha(y)]_{\mathfrak{g}}
%\]
%\end{definition}
\begin{definition}
For any nonnegative integer $k$, a linear map $ d :\mathfrak{g} \to \mathfrak{g}$ is called a $\beta^{k}$-derivation of the involutive Hom-Lie algebra $(\mathfrak{g}, [\cdot,\cdot]_{\mathfrak{g}},\beta)$, if
\begin{eqnarray}
[d,\beta]&=& 0, \text{ that is, } d \circ \beta^{k} = \beta^{k} \circ d, \\
\forall u,v \in \mathfrak{g}: \quad
d[u,v]_{\mathfrak{g}} &=&
[d(u), \beta^{k}(v)]_{\mathfrak{g}}+[\beta^{k}(u),d(v)]_{\mathfrak{g}}.
\end{eqnarray}
\end{definition}
\begin{example}
Let $(\mathfrak{g}, [\cdot,\cdot]_{\mathfrak{g}},\alpha)$ be an involutive multiplicative Hom-Lie algebra. For $x \in \mathfrak{g}$, let consider $\alpha(x)=x$, then $ad_x : \mathfrak{g} \to \mathfrak{g}$ defined by $ad_x (y) = [x,y]_{\mathfrak{g}}$ for all $y\in \mathfrak{g}$ is an $\alpha$-derivation of $(\mathfrak{g}, [\cdot,\cdot]_{\mathfrak{g}},\alpha)$.
\end{example}

%%%%%%%%%%%%%%%%%%%%%%%%%%%%%%%%%%%%%%%%%%%%%%%%%%%%%%%%%%%%%%%%%%%%%%%%%%%%%%%%%%%%%%%%%%%%%%%%%%%%%

\section{HNN-extension of involutive Hom-associative algebras} \label{sec:HNNextensioninvolHomassocalgs}
Let $(A, \ast_{A} ,\alpha_A)$ be an involutive Hom-associative algebra over ring of integers. Let $(B_i, \ast_A, \alpha_{A|_{B_i}})$ $(i \in I)$ be a family of Hom-associative subalgebras of $A$ as defined in Definition 1 (g), with injective morphisms $\theta_i : B_i \to A$, and for each $i \in I$,  a $\theta_i$-derivation $\delta_i: B_i \to A$ such that $\alpha_A$ commutes with $\theta_i$ and $\delta_i$. The associated HNN-extension is presented as  \[H = \langle A,B_i,t_i, \delta_i, \theta_i : i \in I \rangle,\] which is an involutive Hom-associative algebra $H:=(A\cup \{t_i\},\ast_H, \alpha_H)$ in such a way that $x \ast_H y=\alpha_H({x \ast_A y})$, where $\alpha_H(t_i)=t_i$ and $\alpha_H(a)=\alpha_A(a)$ along with a homomorphism $\phi : (A, \ast_A, \alpha_A) \to (H, \ast_H, \alpha_H)$ with the following conditions:
\begin{enumerate}[label=\upshape{(\roman*)},left=0pt]
 \item \label{cond1} $t_i \ast_H (\phi(b)) - \phi(\theta_i(b)) \ast_H t_i=\phi(\delta_i(b))$ for all $b \in B_i$ and all $i \in I$.
 \item \label{cond2} Given any involutive Hom-associative algebra $(S, \ast_S, \alpha_S)$ with elements $\sigma_i \in S$ satisfying
 $\alpha_S(\sigma_i)=\sigma_i$, a morphism $f: (A,\alpha_A) \to (S, \alpha_S)$ such that $\sigma_i \ast_S \alpha_S(f(b)) - \alpha_S (f(b)) \ast_S \sigma_i = f(\delta_i(b))$ for all $b\in B_i$ and $i \in I$, there exists a unique morphism $\theta: (H, \ast_H, \alpha_H) \to (S,\ast_A, \alpha_A)$ such that $\theta(t_i)=\sigma_i$ and $\theta(\phi(a))=f(a)$ for all $a\in A$.
 \end{enumerate}
Assume a single letter $t$ in the condition \ref{cond1} of construction of HNN-extension of involutive multiplicatve Hom-associative algebra. Since $\delta$ is an $\alpha_{A}$-derivation,
\begin{align*}
  \delta(\alpha_A(b))&=t \ast_{H} \alpha_A(b) - \alpha_A(b) \ast_{H} t\\
                     &= \alpha_{H}({ t \ast_A \alpha_A(b)}) - \alpha_{H}({ \alpha_A(b) \ast_A t}) \quad \text{(by definition of $\ast_{H}$)}\\
                     &={\alpha_{H}(t) \ast_A \alpha_{A}^{2}(b)} - { \alpha_{A}^{2}(b) \ast_A \alpha_{H}(t)} \quad \text{(by Def. \ref{def-involutive} (c), (d))} \\
                     &={t\ast_A b}-{b \ast_A t} = \alpha_A(\delta(b)),
\end{align*}
which implies that in the construction of HNN-extension for the case of involutive Hom-associative algebras, it is essential to consider the multiplicative property. It is worth pointing out that the second property of $\alpha$-derivations in Definition \ref{derivation} is straightforward by Hom-associativity.

A left Hom-$B_i$-module ${A}/{B_i}$ is a Hom-module $({A}/{B_i}, \alpha_{A/B_i})$ that comes equipped with a left $B_i$-action,
$ B_i \otimes {A}/{B_i} \to {A}/{B_i},$ with $b \ast_{A/B_i} (a+B_i) =(b \ast_{A} a) +B_i$ and $ \alpha_{A/B_i}: A/B_i \to A/B_i$ with $\alpha_{A/B_i} (a+B_i) = \alpha_A(a) +B_i$, for all $b \in B_i$. Let $X_i$ be a free basis of free left Hom-$B_i$-module $A/B_i$. We define a normal sequence as
 \[(t_{i_{1}} \ast_A \alpha_A(x_1)) \ast_A (t_{i_{2}} \ast_A \alpha_A(x_2)) \ast_A \dots \ast_A (t_{i_{r}} \ast_A \alpha_A(x_r)), \] with $i_{j} \in I$ and $x_\alpha \in X_{i_j}$ for $1 \leq \alpha \leq r$. The set of all normal sequences is denoted by $V$.

Theorem \ref{EmbedingHomAsso} concerns the embeddability of involutive Hom-asso\-ciative algebra into its HNN-extension. We follow the Lichtman and Shirvani's approach \cite{LichtmanShirvani1997:HNNextLiealg} in order to prove that.

\begin{theorem}\label{EmbedingHomAsso}
Let $(A, \ast_A, \alpha_A)$ be an involutive Hom-associative algebra over ring of integers, $B_i$ a family of Hom-associative subalgberas, with injective homomorphisms $\theta_i: B_i \to A$, a $\theta_i$-derivations  $\delta_i: B_i \to A$. Assume that ${A}/{B_i}$ is a free left Hom-$B_i$-module for all $i$, and let $(H,\phi)$ be the corresponding HNN-extension as above. Then the map $\phi$ is an embedding of $A$ into $H$.
\end{theorem}
\begin{proof}
Let us consider the free left Hom-$A$-module on the set of normal sequences, $V$, and denote it by
$$ Q=(\oplus_{u \in V} Au, \alpha_Q), \quad \alpha_Q(u_1,\dots,u_r)=(\alpha_H(u_1),\dots,\alpha_H(u_r)).$$
Consider the morphism of $(A,\alpha_A)$ into $\text{S = (End}_{\mathbb{Z}}(Q),\alpha_S)$ mapping $a \in A$ to left multiplication by $a$ on every factor denoted by $a \mapsto \bar{a}$ and $\alpha_S=\alpha_A$. In the sequel, we need to define suitable $\sigma_i \in S$ for all $i \in I$. If $q \in Q$ is written as
\begin{align*}
             q = \sum_{u \in V} \sum_{x \in X_i} {(b_{x,u} \ast_{{A}/{B}} x) \ast_A u}&=\sum_{u \in V} \sum_{x \in X_i}{(b_{x,u} \ast_{A} \alpha_A(x)) \ast_A u}\\
             & =\sum_{u \in V} \sum_{x \in X_i}{b_{x,u} \ast_{A} (\alpha_A(x) \ast_A u)}
\end{align*}
for $b_{x,u} \in B_i $, define
\[\sigma_i(q)= \sum_{u \in V} \sum_{x \in X_i} ({ \theta_i(b_{x,u}) \ast_A ((t_i \ast_A \alpha_A(x)) \ast_A u)} + {\delta_i(b_{x,u}) \ast_A (\alpha_A(x) \ast_A u)}).\]
We have $\sum_{x\in X_i} (\delta_i(b_{x,u}) \ast_A \alpha_A(x)) \in A$ and every ${((t_i \ast_A \alpha_A(x))\ast_A u)} \in V$. For any element $b \in B_i$ $(i \in I)$, we recall that the left multiplication by $b$ is denoted by $\bar{b}$, so we have
\begin{align*}
{\sigma_i( \bar{b}(q))}&= \sigma_i(\sum_{u\in V} \sum_{x \in X_i} {((b \ast_B b_{x,u}) \ast_A (\alpha_A(x) \ast_A u)}) )\\
                                      &= \sum_{u,x} ({\theta_i(b \ast_B b_{x,u}) \ast_A ((t_i \ast_A \alpha_A(x)) \ast_A u)})\\
                                      &+ \sum_{u,x} {(\delta_i(b \ast_B b_{x,u}) \ast_A (\alpha_A(x) \ast_A u})),
\end{align*}
and
\begin{align*}
    {\overline{\theta_i(b)} (\sigma_i(q))} &= \sum_{i} (\theta_i(b)) \ast_A (\sum_{u \in V} \sum_{x \in X_i} ({ \theta_i(b_{x,u}) \ast_A ((t_i \ast_A \alpha_A(x)) \ast_A u)})) \\
    & + \sum_{i} { (\theta_i(b)) }\ast_A (\sum_{u \in V} \sum_{x \in X_i} {(\delta_i(b_{x,u}) \ast_A (\alpha_A(x) \ast_A u)})).
\end{align*}

Hence, \[ \sigma_i (\bar{b}(q)) - \overline{\theta_i(b)}( \sigma_i (q))=\sum_{u,x} { ((\delta_i(b) \ast_A b_{x,u}) \ast_A (\alpha_A(x)\ast_A u))}= {\overline{\delta_i(b)}(q)}.\] Therefore, the property (2) implies that there exists $\theta: (H, \ast_H, \alpha_H) \to (S, \ast_S, \alpha_S)$ such that $\theta(t_i)=\sigma_i$ and $\theta (\phi(a))=\bar{a}$ for all $a \in A$.
\end{proof}

%%%%%%%%%%%%%%%%%%%%%%%%%%%%%%%%%%%%%%%%%%%%%%%%%%%%%%%%%%%%%%%%%%%%%%%%%%%%%%%%%%%%%%%%%%%%%%%%%%%%%
\section{HNN-extension of involutive Hom-Lie algebras}
\label{sec:HNNextensioninvolHomLiealgs}
Let $(A, {\ast_{A}}, {\alpha}_{A})$ be an arbitrary Hom-associative algebra, and let $(A, {[\cdot,\cdot]}_{A}, {\beta}_{A})$ be the Hom-Lie algebra defined by
\[ {[x,y]}_{A} = x {\ast}_{A} y - y {\ast}_{A} x,\]
and $ {\beta}_{A} = {\alpha}_{A}$, for $x,y \in A$. If $(\mathfrak{g},{[\cdot,\cdot]}_{\mathfrak{g}},\beta_{\mathfrak{g}})$ is an involoutive Hom-Lie algebra, then  $({\mathfrak{U}}_{\mathfrak{g}} ,  {\phi}_{g})$ is called a universal enveloping Hom-associative algebra of $\mathfrak{g}$, if
$$ {\phi}_{\mathfrak{g}} : (\mathfrak{g},[\cdot,\cdot]_{\mathfrak{g}},\beta_{\mathfrak{g}}) \to ({\mathfrak{U}}_{\mathfrak{g}}, {[\cdot,\cdot]}_{\mathfrak{U}_{\mathfrak{g}}}, {\beta}_{\mathfrak{U}_{\mathfrak{g}}})$$
is a homomorphism of Hom-Lie algebras,
\[ \phi_{\mathfrak{g}}([x,y]_{\mathfrak{g}}) = [\phi_{\mathfrak{g}}(x),\phi_{\mathfrak{g}}(y)]_{{\mathfrak{U}}_{\mathfrak{g}}}, \qquad
{\phi}_{\mathfrak{g}} ( {\beta}_{\mathfrak{g}}(x))={\beta}_{\mathfrak{U}_{\mathfrak{g}}} ({\phi}_{\mathfrak{g}} (x)), \]
satisfying the following universal property: for any involutive Hom-associative algebra $A = (A, {{\ast}_{A}}, {\alpha}_{A})$ and any Hom-Lie algebra morphism $\varepsilon : (\mathfrak{g}, [\cdot,\cdot]_\mathfrak{g}, {\beta}_{\mathfrak{g}}) \to (A, [\cdot,\cdot]_{A}, \beta_{A})$, there exists a unique morphism $\eta: {\mathfrak{U}}_{\mathfrak{g}} \to A$ of Hom-associative algebras such that $\eta {\phi}_{\mathfrak{g}} = \varepsilon $.  For any involutive Hom-Lie algebra there exists a universal enveloping Hom-associative algebra, which is involutive and Poincare-Birkhoff-Witt theorem is valid for it. This shows that the map ${\phi}_{\mathfrak{g}}$ is injective, and we can say that every $\beta_{\mathfrak{g}}$-derivation of involutive Hom-Lie algebra $(\mathfrak{g}, [\cdot,\cdot]_\mathfrak{g}, {\beta}_{\mathfrak{g}})$ extends to ${\beta}_{\mathfrak{U}_{\mathfrak{g}}}$-derivation of ${\mathfrak{U}}_{\mathfrak{g}}$.

\begin{definition} \label{HNN-Hom-Lie}
Let $(\mathfrak{g},[\cdot,\cdot]_\mathfrak{g},\beta_\mathfrak{g})$ be an involutive Hom-Lie algebra and $\mathfrak{s}$ be a subalgebra. Assume that $d: \mathfrak{s} \to \mathfrak{g}$ is a $\beta_\mathfrak{g}$-derivation. The associated HNN-extension is given by the following presentation
\[ \mathfrak{h} := \langle \mathfrak{g} , t: d(s)=[t,s]_{\mathfrak{h}}, s\in \mathfrak{s} \rangle,\]
which is an involutive Hom-Lie algebra $(\mathfrak{h}, [\cdot,\cdot]_\mathfrak{h}, {\beta}_{\mathfrak{h}})$ with ${\beta}_{\mathfrak{h}}(t)=t$, ${\beta}_{\mathfrak{h}}(g)= {\beta}_{\mathfrak{g}}(g)$ for $g \in \mathfrak{g}$. This means that the presentation of $\mathfrak{g}$ is augmented by adding a new generating symbol $t$, and for each $s \in \mathfrak{s}$, the relation $[t,s]_{\mathfrak{h}}=d(s)$ is added. We note that $[g_1,g_2]_\mathfrak{h}=[g_1,g_2]_\mathfrak{g}$, for all $g_1,g_2 \in \mathfrak{g}$.
\end{definition}
Let assume that in the Definition \ref{HNN-Hom-Lie}, $\mathfrak{s}=\mathfrak{g}$, therefore, $d$ is a $\beta_\mathfrak{g}$-derivation of $\mathfrak{g}$ and $\mathfrak{h}$ is then the semi-direct product of $\mathfrak{g}$ with a one-dimensional involutive Hom-Lie algebra which acts on $\mathfrak{g}$ via $d$. In order to make this special case more clear, we recall the concepts of Hom-action and semidirect product of Hom-Lie algebras in the sequel in accordance with \cite{CasasGarciaMartinezJPAA2020:abelextscrossmodHomLiealg}.
\begin{definition}
Let $(\mathfrak{l},\alpha_{\mathfrak{l}})$ and $(\mathfrak{m},\alpha_{\mathfrak{m}})$ be Hom-Lie algebras. A Hom-action from $(\mathfrak{l},\alpha_{\mathfrak{l}})$ on $(\mathfrak{m},\alpha_{\mathfrak{m}})$ is expressed by a bilinear map $$\sigma:\mathfrak{l} \otimes \mathfrak{m} \to \mathfrak{m}, \quad \sigma(x \otimes m)= x \sq m$$ such that
\begin{enumerate}[label=\upshape{(\alph*)},left=0pt]
    \item  \quad $[x,y] \sq \alpha_{\mathfrak{m}} (m) = \alpha_{\mathfrak{l}}(x) \sq (y \sq m) - \alpha_{\mathfrak{l}}(y) \sq (x \sq m)$,
    \item \quad $\alpha_{\mathfrak{l}}(x) \sq [m,m^\prime]=[x \sq m, \alpha_{\mathfrak{m}}(m^\prime)]+[\alpha_{\mathfrak{m}}(m), x \sq m^{\prime}]$,
    \item \quad $\alpha_{\mathfrak{m}}(x \sq m)= \alpha_{\mathfrak{l}}(x) \sq \alpha_{\mathfrak{m}}(m)$,
\end{enumerate}
for all $x,y \in \mathfrak{l}$ and $m,m^{\prime} \in \mathfrak{m}$.
\end{definition}

\begin{definition}[\cite{CasasGarciaMartinezJPAA2020:abelextscrossmodHomLiealg}]
Let $(\mathfrak{l},\alpha_{\mathfrak{l}})$ and $(\mathfrak{m},\alpha_{\mathfrak{m}})$ be Hom-Lie algebras with an action from $(\mathfrak{l},\alpha_{\mathfrak{l}})$ on $(\mathfrak{m},\alpha_{\mathfrak{m}})$. The semidirect product $(\mathfrak{m} \rtimes \mathfrak{l} , \Tilde{\alpha}) $ is the Hom-Lie algebra with underlying $K$-vector space $\mathfrak{m} \oplus \mathfrak{l}$, with bracket
\[ [(m_1,x_1),(m_2,x_2)]=([m_1,m_2]+ x_1 \sq m_2 - x_2 \sq m_1, [x_1,x_2])\]
and endomorphism
\[ \Tilde{\alpha}: \mathfrak{m} \oplus \mathfrak{l} \to \mathfrak{m} \oplus \mathfrak{l},  \qquad \Tilde{\alpha}(m,x)=(\alpha_{\mathfrak{m}}(m), \alpha_{\mathfrak{l}}(x)) \]
for all $x,x_1,x_2 \in \mathfrak{l}$ and $m,m_1,m_2 \in \mathfrak{m}$.
\end{definition}

If in the Definition \ref{HNN-Hom-Lie} of HNN-extension of involutive Hom-Lie algebras, $\beta_\mathfrak{g}$-deri\-vation map is defined on the whole involutive Hom-Lie algebra $(\mathfrak{g},[\cdot,\cdot]_\mathfrak{g},\beta_{\mathfrak{g}})$, then a semidirect product of
one-dimensional involutive Hom-Lie algebra with $\mathfrak{g}$ with respect to $\beta_\mathfrak{g}$-derivation map will be obtained.

\begin{theorem}
Any involutive Hom-Lie algebra embeds into its HNN-extension.
\end{theorem}
\begin{proof}
Let $({\mathfrak{U}}_{\mathfrak{g}} ,  {\phi}_{g})$ and $({\mathfrak{U}}_{\mathfrak{s}} ,  {\phi}_{s})$ be the universal enveloping Hom-associative algebras corresponding to, respectively, the involutive Hom-Lie algebra $\mathfrak{g}$ and its subalgebra $\mathfrak{s}$, which are involutive with respect to  Lemma \ref{invouniversal}. Let $\mathfrak{h}=\langle \mathfrak{g} , t: d(s)=[t,s]_{\mathfrak{h}}, s\in \mathfrak{s} \rangle $ be the HNN-extension of involutive Hom-Lie algebra $(\mathfrak{g},[\cdot,\cdot]_\mathfrak{g},\beta_\mathfrak{g})$ as above. By extending $d$ to a $\beta_{{\mathfrak{U}}_{\mathfrak{g}}}$-derivation of ${\mathfrak{U}}_{\mathfrak{g}}$ defined on ${\mathfrak{U}}_{\mathfrak{s}}$ we form the HNN-extension of involutive Hom-associative algebra ${\mathfrak{U}}_{\mathfrak{g}}$ which is denoted by $M=\langle {\mathfrak{U}}_{\mathfrak{g}}, {\mathfrak{U}}_{\mathfrak{s}},t, \delta \rangle $. Let $(R,\ast_R, \alpha_R)$ be an arbitrary involutive Hom-associative algebra with a homomorphism of Hom-Lie algebras $(\mathfrak{h}, [\cdot,\cdot]_\mathfrak{h}, {\beta}_{\mathfrak{h}}) \to  (R, {[\cdot,\cdot]}_{R}, {\beta}_{R})$. The restriction to $\mathfrak{g}$ extends to a homomorphism ${\mathfrak{U}}_{\mathfrak{g}} \to R$, which extends to a homomorphism $M \to R$, so we have ${\mathfrak{U}}_{\mathfrak{h}} \simeq M$.  As $ {\mathfrak{U}}_{\mathfrak{g}}/{\mathfrak{U}}_{\mathfrak{s}} $ is a free left Hom-${\mathfrak{U}}_{\mathfrak{s}}$-module, Theorem \ref{EmbedingHomAsso} implies that ${\mathfrak{U}}_{\mathfrak{g}}$ is embedded into $M$, and so $\mathfrak{g}$ embeds into its HNN-extension.
\end{proof}

\section*{Acknowledgement}
Chia Zargeh was supported by postdoctoral scholarship CNPq, Conselho Nacional de Desenvolvimento Científico  e Tecnológico - Brasil (152453/2019-9).

%\newpage
%\bibliographystyle{amsplain}

\end{document}